\documentclass[12pt]{amsart}
\usepackage{amsmath, amssymb, amsthm, latexsym}
\usepackage{amssymb}
\usepackage{latexsym}
\usepackage[center]{caption}
\usepackage{tikz}
\newcounter{braid}
\newcounter{strands}
\def\holim{\protect\displaystyle\operatornamewithlimits{holim}_\leftarrow}

\pgfkeyssetvalue{/tikz/braid height}{1cm}
\pgfkeyssetvalue{/tikz/braid width}{1cm}
\pgfkeyssetvalue{/tikz/braid start}{(0,0)}
\pgfkeyssetvalue{/tikz/braid colour}{black}
\pgfkeys{/tikz/strands/.code={\setcounter{strands}{#1}}}

\makeatletter
\def\cross{%
  \@ifnextchar^{\message{Got sup}\cross@sup}{\cross@sub}}

\def\cross@sup^#1_#2{\render@cross{#2}{#1}}

\def\cross@sub_#1{\@ifnextchar^{\cross@@sub{#1}}{\render@cross{#1}{1}}}

\def\cross@@sub#1^#2{\render@cross{#1}{#2}}

\def\render@cross#1#2{
  \def\strand{#1}
  \def\crossing{#2}
  \pgfmathsetmacro{\cross@y}{-\value{braid}*\braid@h}
  \pgfmathtruncatemacro{\nextstrand}{#1+1}
  \foreach \thread in {1,...,\value{strands}}
  {
    \pgfmathsetmacro{\strand@x}{\thread * \braid@w}
    \ifnum\thread=\strand
    \pgfmathsetmacro{\over@x}{\strand * \braid@w + .5*(1 - \crossing) * \braid@w}
    \pgfmathsetmacro{\under@x}{\strand * \braid@w + .5*(1 + \crossing) * \braid@w}
    \draw[braid] \pgfkeysvalueof{/tikz/braid start} +(\under@x pt,\cross@y pt) to[out=-90,in=90] +(\over@x pt,\cross@y pt -\braid@h);
    \draw[braid] \pgfkeysvalueof{/tikz/braid start} +(\over@x pt,\cross@y pt) to[out=-90,in=90] +(\under@x pt,\cross@y pt -\braid@h);
    \else
    \ifnum\thread=\nextstrand
    \else
     \draw[braid] \pgfkeysvalueof{/tikz/braid start} ++(\strand@x pt,\cross@y pt) -- ++(0,-\braid@h);
    \fi
   \fi
  }
  \stepcounter{braid}
}

\tikzset{braid/.style={double=\pgfkeysvalueof{/tikz/braid colour},double distance=1pt,line width=2pt,white}}

\newcommand{\braid}[2][]{%
  \begingroup
  \pgfkeys{/tikz/strands=2}
  \tikzset{#1}
  \pgfkeysgetvalue{/tikz/braid width}{\braid@w}
  \pgfkeysgetvalue{/tikz/braid height}{\braid@h}
  \setcounter{braid}{0}
  \let\sigma=\cross
  #2
  \endgroup
}
\makeatother

\input xypic
\newtheorem{theorem}{Theorem}
\newtheorem{proposition}[theorem]{Proposition}

\newtheorem{definition}[theorem]{Definition}

\def\Proof{\medskip\noindent{\bf Proof: }}

\def\Z{\mathbb{Z}}

\def\Pi{\mathbb{P}^{\infty}}

\def\md{\mathcal{D}}

\def\mt{\mathcal{T}}

\def\qed{\hfill$\square$\medskip}

\def\Zpk{\mathbb{Z}/p^{k}}
\def\Zpk1{\mathbb{Z}/p^{k-1}}

\newcommand{\rref}[1]{(\ref{#1})}

\newcommand{\beg}[2]{\begin{equation}\label{#1}#2\end{equation}}
\def\r{\rightarrow}

\def\me{\mathcal{E}}

\def\mf{\mathcal{F}}

\def\sl2{\widetilde{SL_{2}(\Z)}}

\def\md{\mathcal{D}}

\title[Equivariant cobordism ring]{The equivariant complex cobordism ring of a finite abelian group}
\author{William C. Abram, Igor Kriz}
\address{Department of Mathematics \\ Hillsdale College \\ Hillsdale, MI 49242}
  \email{wabram@hillsdale.edu}
\address{Department of Mathematics \\ University of Michigan \\ Ann Arbor, MI 48109}
    \email{ikriz@umich.edu}
\thanks{The first author was supported by an NSF graduate fellowship. 
The second author was supported by NSF grant DMS 1102614}

\begin{document}

\maketitle

\begin{abstract}
We compute the equivariant (stable) complex cobordism ring $(MU_G)_*$ for finite abelian groups $G$.
\end{abstract}

\section{Introduction}

\vspace{3mm}

The calculation of the non-equivariant cobordism ring due to Milnor and Quillen \cite{milnor,q} was one of the
great successes of algebraic topology. 
The $G$-equivariant complex cobordism ring  for $G$ a compact Lie group can be
defined analogously to the non-equivariant case. It was noticed almost immediately however
(e.g. \cite{w}) that because of failure of equivariant
transversality, equivariant cobordism groups are not the homotopy groups of an $RO(G)$-graded 
generalized (co)homology theory and hence are much more difficult
to calculate. 
Because of this, tom Dieck \cite{td} 
introduced the {\em stable} equivariant complex cobordism ring, which is the universal object
remedying this situation. It has both a geometric characterization (Br\"{o}cker and Hook \cite{bh}) 
and a characterization as
the coefficient ring of the $G$-equivariant Thom spectrum.

\vspace{3mm}
Perhaps surprisingly, the problem of calculating
explicitly tom Dieck's stable equivariant cobordism ring $(MU_G)_*$
has remained open for the last 40 years, despite some great progress (\cite{cm, gm, Greenlees}, just to name a few
milestones).
To date, there were only two complete calculations known: 
The case of a $p$-primary cyclic group was done by the second author \cite{k}. This computation
comes in the form of a pullback diagram, but a recipe is given in \cite{k} for
recovering explicitly individual elements of the cobordism ring from the diagram. This 
method was used by Strickland \cite{strick}
to give, by purely algebraic methods, a presentation of the $\Z/2$-equivariant stable cobordism ring in terms of
commutative ring generators and defining relations. 

The other known computation is due to Dev Sinha \cite{sinha}. His result is a beautiful
presentation of the $MU_*$-algebra $(MU_{S^1})_*$ in terms of generators and
defining relations. This computation, in fact, has the additional benefit that it gives explicit algebra
generators of $(MU_{(S^1)^n})_*$, and via a surprising short exact sequence,
also generators of $(MU_G)_*$ for any finite abelian group $G$. Sinha's approach uses Comezana's
theorem \cite{comezana} that $(MU_G)_*$ for a compact Lie group is a {\em free $MU_*$-module}. 
This is used to pick splittings of restriction maps. Comezana's proof is 
highly non-constructive,
and Sinha's generators are therefore, necessarily, non-explicit. From the point of
view of \cite{k,strick}, it is, for example, not even at all obvious how to write down explicit free generators of
$(MU_{\Z/2})_*$ as an $MU_*$-module. What is remarkable about the main theorem of
\cite{sinha} about $(MU_{S^1})_*$ is that changing generators within the choices 
allowed leads to an isomorphism of rings {\em with relations of the same form}. It is worth noting that there
is a connection between the results of Sinha and Strickland, as Strickland's generators satisfy the 
defining properties of Sinha's generators.

\vspace{3mm}
The 
main result of the present note is an explicit calculation of $(MU_G)_*$ for a finite
abelian group $G$. While
the meaning of the words ``explicit calculation'' 
is debatable in the case of a complicated ring such as $(MU_G)_*$, the answer we give here 
is purely algebraic, described in terms of concrete ring-theoretic constructions. 
In fact, the form in which the result appears is a direct generalization of \cite{k}, 
with the pullback replaced by a more complicated limit diagram. Similar
comments as in \cite{k} regarding extracting specific elements 
apply to the present case, and the method of Strickland \cite{strick}
can therefore in principle also be applied to our present situation. For future directions, one can imagine
geometric applications, for example to studying group actions on manifolds with isolated fixed sets or the study of 
equivariant genera, as Sinha did, or applications to equivariant formal group laws, or generalizations to other
complex-oriented equivariant cohomology theories.

\vspace{3mm}
The effective methods of this paper build upon a rich heritage of localization and completion techniques, 
which facilitate computations that would otherwise prove intractable. Early examples would be the Atiyah-Segal 
completion theorem in equivariant $K$-theory \cite{AS69} and tom Dieck's localization theorems for equivariant 
cohomology \cite[Chapter 7]{tD79}. Another early treatment of equivariant localization can be found in \cite{MMT82}. 
Localization and completion theorems for $MU$-module spectra can be found in \cite{gm}. Many of the techniques that 
we employ, including the use of families, will be familiar to the experts. A basic reference with good exposition would be 
the ``Alaska" volume \cite{May96}. For an introduction to equivariant cobordism, including the use of families, 
see \cite[Chapter XV]{May96}. For localization and completion results and calculations in complex cobordism, see 
\cite[Chapters XXV-XXVI]{May96}. The computations of this paper are accomplished by use of the Tate square, as in \cite{k}, 
and in particular we apply the methods of that paper inductively. 

\vspace{5mm}
\section{Statement of the main result}
\label{sstaste}
We must first recall certain basic concepts of equivariant homotopy theory
(\cite{lms}). Recall that a {\em family} $\mathcal{F}$ of subgroups of a finite group is
a system closed under subgroups and conjugation (the latter being vacuous in the abelian
case). The {\em classifying space} of a family $\mathcal{F}$ is a $G$-CW complex $E\mathcal{F}$
which satisfies 
$$E\mathcal{F}^H\simeq\left\{\begin{array}{ll}
* & \text{if $H\in\mathcal{F}$}\\
\emptyset & \text{else}.
\end{array}\right.
$$
Recall also the homotopy cofiber sequence
$$E\mathcal{F}_+\r S^0\r \widetilde{E\mathcal{F}},$$
where the subscript $+$ means the inclusion of a disjoint base point. 
We will mostly be interested in two kinds of families associated with a subgroup $H\subseteq G$,
namely the family $\mf(H)$ of subgroups contained in $H$ and the family $\mf[H]$ of subgroups
not containing $H$. Instead of $E\mf(H)$, one usually writes $E(G/H)$.

Let $G$ be a finite abelian group. 
Denote by $P(G)$ the poset of all non-empty sets $S$ of subgroups of $G$ which are totally
ordered by inclusion:
\beg{e1}{S=\{H_1\subsetneq H_2\subsetneq...\subsetneq H_k\},
}
with ordering given by inclusion: $S\leq T$ if and only if $S\subseteq T$. For example, $G=\Z/4$ has $3$ subgroups
totally ordered by inclusion, so
this poset consists of all non-empty subsets of the set of subgroups, and has $7$ elements.

Let $X$ be a $G$-equivariant spectrum. In this note, we only consider $G$-equivariant
spectra  indexed over a complete universe - see \cite{lms}.

\begin{definition}
\label{e1a}
Define a functor
$$\Gamma=\Gamma_{G,X}:P(G)\r \text{$G$-spectra}$$
as follows: Denoting the $S$ from \rref{e1} as $S_k$ (to capture the number of elements), define inductively:
\beg{e1b1}{\Gamma(S_1)=F(E(G/H_{1})_+,E\widetilde{\mf[H_1]}\wedge X),}
\beg{e1b}{\begin{array}{l}
\Gamma(S_k)=
F(E(G/H_{k})_+,E\widetilde{\mf[H_k]}\wedge\Gamma(S_{k-1})).
\end{array}
}
Note that there is a canonical and natural morphism of $G$-spectra
\beg{e2}{Y\r F(E(G/H)_+,\widetilde{E\mf[H]}\wedge Y),
}
and the effect of $\Gamma$ on arrows is defined by iterating these maps. 
\end{definition}

By iterating 
\rref{e2}, there is also a canonical natural transformation
\beg{e2a}{\gamma_X:Const_X\r \Gamma
}
where $Const_X$ is the constant functor on $P(G)$ with value $X$. Note that in the case of $G=\Z/p$
for $p$ prime, this definition reduces to the ``Tate square" of Greenlees and May \cite{gmt}.

In the next section, we shall calculate the effect of the functor $\Gamma$ on coefficients
explicitly in the case $X=MU_G$.
This is relatively routine, although the statement is technical. Our main result is the
following:

\begin{theorem}
\label{t1}
For $X=MU_G$,
applying the coefficient (homotopy groups) functor to the map $\gamma_{MU}$ of \rref{e2a} induces an isomorphism
\beg{et1}{\diagram (MU_{G})_*\rto^(.4)\cong & \displaystyle\lim_\leftarrow \Gamma(S)_*.
\enddiagram}
\end{theorem}

Since we are dealing with an inverse limit, the validity
of the isomorphism in the category of abelian groups automatically implies its
validity in the category of commutative rings.

\vspace{3mm}
\begin{proposition}
\label{p0}
The limit on the right hand side of \rref{et1} can be calculated by restricting to the 
subset $P^{\prime}(G)$ of $P(G)$ consisting of sets $S$ of cardinality $\leq 2$,
or even further to the subset $P^{\prime\prime}(G)$
consisting of sets $S$ which have either cardinality $1$ or consist of two groups
$$H_1\subsetneq H_2$$
for which there does not exist a group $K$ which would satisfy
\beg{ep0}{H_1\subsetneq K\subsetneq H_2.}
\end{proposition}

\begin{proof}
Since below every element of $P(G)$ there is an element of the form $\{H\}$, $\displaystyle\lim_\leftarrow \Gamma(S)_*$
is a subgroup of the product of the groups $\Gamma(\{H\})_*$. Since whenever 
$\{H_i\}$, $i=1,2$, $H_1\neq H_2$, are less or equal to
an element $S$ of the poset, there is an inclusion between $H_1$ and
$H_2$, and they are therefore both less or equal than $\{H_1,H_2\}$, the statement about $P^\prime$ follows.

To prove the statement about $P^{\prime\prime}$, it suffices to show that for configurations of the type \rref{ep0},
$\{H_1,H_2\}$ can be omitted from the diagram, provided we have $\{H_1,K\}$, $\{K,H_2\}$. The key point is that if 
we have $\{H_1,K\}$, $\{K,H_2\}$, we can equivalently include also $\{H_1,K,H_2\}$, by functoriality (as in the 
first part of the argument). We will see in Section \rref{s2} below that the structure homomorphisms
$\Gamma(\{H_1\})\r \Gamma(\{H_1,K\})$ and $\Gamma(\{H_1\})\r \Gamma(\{H_1,K,H_2\})$ are injective. Now consider
a diagram in abelian groups, or sets of the form
$$
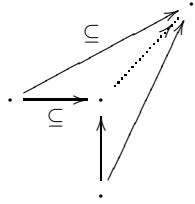
\diagram
&&\cdot\\
\cdot \rto_\subseteq\urrto^\subseteq &\cdot\urdotted|>>\tip\\
&\cdot\uto\uurto &
\enddiagram$$
Then the pullback of the diagram comprised of the horizontal and vertical arrows is isomorphic, via
the canonical map, to the pullback of the diagram comprised 
of the diagonal arrows because the dotted arrow is injective on the image of the horizontal one. Letting the dotted
arrow be the structure map $\Gamma(\{H_1,H_2\})\r \Gamma(\{H_1,K,H_2\})$ gives the required result.
\end{proof}

\vspace{5mm}

\section{Computation of the functor $(\Gamma_{G,MU})_*$}
\label{s2}

\vspace{3mm}
This is essentially a gathering of known facts. Let, for any abelian
group $A$, $A^*=Hom(A,S^1)$ and $\overline{A}=A\smallsetminus \{0\}$.
Recall that by tom Dieck's
result \cite{td,k}, \cite{Greenlees}, Corollary 10.4, we have
\beg{e8}{\begin{array}{l}(E\widetilde{\mf[H]}\wedge MU)^{H}_{*}=\\
MU_*[u_{L}^{\pm 1},u_{L}^{(i)}|i>0, L\in \overline{H^{*}}].
\end{array}
}
For the purposes
of this note we don't really need to know what the classes $u_{L}^{(i)}$ are, 
(we set $u_{L}^{(0)}=u_L$), the only fact we need to know is that under the canonical
map of \rref{e8} into
\beg{e81}{(\begin{array}{l}\widetilde{E\mf[H]}\wedge F(EG_+,MU))^{H}_{*}=\\
MU_*[[u_L|L\in H^{*}]]/
(u_{L}+_F u_{M}=u_{LM}),
\end{array}}
we have
\beg{eui}{u_{L}^{(i)}\mapsto\parbox{3.5in}{The coefficient of $x^i$ in $x+_F u_L$}}
(see \cite{k}, Theorem 1.1).

\vspace{3mm}
\begin{definition}
\label{de9}
Recall the notation $S$ from equation \rref{e1}. Let
$MU_{S,0} = MU$, and define inductively
\beg{e9}{
MU_{S,j} = (\widetilde{E \mathcal{F}[H_j]} \wedge F(E(G/H_{j-1})_+, MU_{S,j-1})_*^{H_j}.
}
for $j=1,\dots,k$.
\end{definition}

\vspace{3mm}

Now consider $H_1=H$ in \rref{e8}, \rref{e81}.
Assume inductively that we have calculated the coefficients of the $H_{j-1}$-spectrum $MU_{S,j-1}$.
The $H_j/H_{j-1}$-spectrum \rref{e9} is split only if $j=2$, but in either case
the Borel cohomology spectral sequence associated with
\beg{e11}{F(E(G/H_{j-1})_+,MU_{S,j-1})^{H_j}_{*}
}
collapses by evenness. Hence we know \rref{e11} has an associated graded
object isomorphic to
\beg{e12}{(MU_{S,j-1})^{H_{j-1}}_{*}B(H_j/H_{j-1}).
}
On the other hand, the coefficient ring \rref{e11} is generated as a $(MU_{S,j-1})_*$-algebra
by Euler classes of $1$-dimensional complex representations, so the precise
relations in the ring \rref{e11} are not difficult to compute from the formal group law.
Furthermore,
\beg{e13}{(\widetilde{E\mf[H_j]}\wedge F(E(G/H_{j-1})_+,MU_{S,j-1}))^{H_j}_{*}
}
is obtained from \rref{e11} by inverting the Euler classes $u_L$ of irreducible complex representations
$L$ of $H_j$ which are non-trivial on $H_j$.

\vspace{3mm}
Explicitly, let $R_j$, $j=0,...,k$ be a set of $G/H_{j}$-representatives of the irreducible
non-trivial complex $H_{j+1}/H_j$-representations (we set $H_0=\{e\}$, $H_{k+1}=G$).

\vspace{3mm}
\begin{definition}
\label{dering}
Let $A_{G,S}=A_S$ denote the ring
\beg{eering}{\begin{array}{l}
MU_*[u_L,u_{M}^{-1},u_{N}^{(i)}|i>0, \\
L\in R_0\amalg...\amalg R_{k}, M\in R_0\amalg ...\amalg R_{k-1}, N\in R_0].
\end{array}}
\end{definition}

\vspace{3mm}
On this ring, define the following topology $\mt_{G,S}=\mt_S$, similar to topologies which often
occur in completion theorems: A sequence of monomials
$$a_t\prod_{L\in R_1\amalg...\amalg R_k} u_{L}^{n(L,t)}\in A_S,\; t=1,2,,\dots
$$
with
$$0\neq a_t\in MU_*[u_{L}^{\pm1},u_{L}^{(i)}|i>0,L\in R_0]$$
converges to $0$ if and only if there exists a $j=1,...,k$ such that the
following two conditions are met:

\vspace{3mm}
\noindent
(A) $n(L,t)$ is eventually constant in $t$ for $L \in R_i$, $i >j$; and \newline
(B) $n(L,t) \rightarrow_t +\infty$ for $L \in R_j$.
for $L\in R_j$. 

\vspace{3mm}
A sequence of elements $p_t\in A_{S}$ converges to $0$ if and only if
choosing arbitrary non-zero monomial summands $m_t$ of $p_t$, the sequence of
monomials $m_t$ converges to $0$ in $t$. A set $T\subset A_S$ is closed if and only if
the limit of every sequence in $T$ convergent  in $A_S$ is in $T$.

\vspace{3mm}

\begin{theorem}
\label{tc}
$\Gamma(S)_*$ is the quotient of the completion
$$(A_S)^{\wedge}_{\mt_S}$$
by the (closed) ideal $I_S=I_{G,S}$ generated by the relations
\beg{etc1}{u_{L_1}+_F u_{L_2}=\left(\sum_{i=1}^{m}\right)_Fu_{M_i}}
whenever 
\beg{etc2}{L_1 L_2\cong\prod_{i=1}^{m} M_i}
and there exists a $j=1,...,k$ such that
$$L_1,L_2\in R_j,$$
$$M_i\in R_j\amalg...\amalg R_k.$$
\end{theorem}

Note: the relation \rref{etc1} of course holds for any $1$-dimensional complex representations $L_1,L_2$
which satisfy \rref{etc2}, but recall that we have restricted attention to a specific set of generators in \rref{eering}.

\Proof
An induction on $|G|$ and $k$, using the method described in the beginning of
this section. For $|G|=1$ or $k=1$ the statement is obvious. For a given $k>1$,
first assume $H_k\neq G$. Then, motivated by the Borel cohomology spectral sequence 
(\cite{gmt}, Theorem 10.5), filter the ring
$$(A_{G,S})^{\wedge}_{\mt_{G,S}}/I_{G,S}$$
by powers of the ideal 
$$(u_L|L\in R_k).$$
By definition, the associated graded ring is
$$((A_{H_k,S})^{\wedge}_{\mt_{H_k,S}}/I_{H_k,S})[[u_L|L\in R_k]]/(u_L+_F u_M=u_{LM})$$
(with the understanding, of course, that $u_0=0$) which, by the induction hypothesis,
coincides with \rref{e12}. The filtration also coincides with the Borel cohomology spectral sequence,
so the statements follows from that spectral sequence. The Borel cohomology spectral
sequence for complex cobordism in the abelian case is quite standard, see e.g. \cite{cm}.

When $H_k=G$, we have, by definition,
$$(A_{G,S})^{\wedge}_{\mt_{G,S}}/I_{G,S}=
(A_{G.S\smallsetminus\{G\}})^{\wedge}_{\mt_{G,S\smallsetminus\{G\}}}/I_{G,S\smallsetminus\{G\}}
[u_{L}^{-1}|L\in R_k],$$
which is $\Gamma_G(S)_*$ by the induction hypothesis and \rref{e8}.
\qed

\vspace{3mm}

It remains to compute the effect of $\Gamma$ on arrows (i.e. inclusions of $S$), but this is given
simply by
$u_L\mapsto u_L$
(i.e by these classes being sent to classes of the same name)
and by \rref{eui}, where applicable. This is a consequence of naturality of Euler classes.
Of course, our description of $\Gamma(S)_*$ depended on choices of $G/H_j$-representatives
of irreducible complex $H_{j+1}/H_j$-representations, so we need to specify how the description
changes when we change representatives. For $j>1$, replacing $L$ by 
$$L^\prime=L\prod_{i=1}^{m}M_i$$
with $M_i\in R_{j+1}\amalg...\amalg R_{k}$, we may simply use the relation
$$u_{L^\prime} =u_L+_F u_{M_1} +_F... +_F u_{M_m}.$$
For $j=1$, we use the relation
$$(u_{L^\prime}+_F x)=u_{L} +_F (u_{M_1}+_F...+_F u_{M_m}+_F x)$$
and compare the coefficients at $x^i$, where the contents of the parenthesis on the right
hand side are expanded as a series in $x$.

\vspace{5mm}
\section{An example: $G = \mathbb{Z}/p^n$}

Let us illustrate the result of these computations on 
the example $G = \mathbb{Z}/p^n$ for a prime $p$. As usual, the notation $[k]_F x$ denotes the 
$k$-fold sum $x +_F x +_F \cdots +_F x$ of $x$ under the formal group law $F$ (in our case, the non-equivariant
universal formal group law). Let $u_{[k]}$ denote $[p^k]_Fu$, and 
\[R_k = MU_*[u_j,u_j^{-1},b_j^{(i)} | i > 0, j \in \{1,2,\ldots,p^k-1\}][[u_{[k]}]]/([p^{n-k}]_Fu_{[k]}),\]
$$\begin{array}{l}S_k = MU_*[u_j,u_j^{-1},b_j^{(i)} | i > 0, j \in \{1,2,\ldots,p^k-1\}][[u_{[k]}]]/([p^{n-k}]_Fu_{[k]})[u_{[k]}^{-1}],
\end{array}$$
\[R^n = MU_*[u_j,u_j^{-1},b_j^{(i)} | i >0, j \in \{1,2,\ldots,p^n-1\}]. \]
Then Theorem \ref{t1} and Proposition \ref{p0} say that
$(MU_{\mathbb{Z}/p^n})_*$ is the $n$-fold pullback of the diagram of rings
\begin{equation} \label{zpnringcompeq} \xymatrix{& & & & R^n \ar[d]^{\phi^{n-1}} \\
& & &  R_{n-1} \ar[r]^{\psi_{n-1}} \ar[d]^{\phi_{n-2}} & S_{n-1} \\
& & R_2 \ar[d]^{\phi_1} \ar[r]^\cdots & S_{n-2} & & \\
& R_1 \ar[d]^{\phi_0} \ar[r]^{\psi_1} & S_1 & & \\
R_0 \ar[r]^{\psi_0} & S_0. & & &\\
\\}
\end{equation}
The maps $\psi_k$ are localization by inverting $u_{[k]}$, and the maps $\phi_k$ are determined by the properties of sending $u_{[k+1]}$ to $[p]_F u_{[k]}$ and $b_j^{(i)}u_j$ to the coefficient of $x^i$ in $x +_F [j]_F u_{[k]}$. $\phi^{n-1}$ is determined by the property of sending $b_j^{(i)}u_j$ to the coefficient of $x^i$ in $x +_F [j]_F u_{[k]}$.

\vspace{5mm}
\section{Proof of the main theorem}

First note that the natural transformation \rref{e2a} gives a canonical
morphism of $G$-spectra
\beg{e3}{\eta_X:X\r \operatornamewithlimits{holim}_\leftarrow \Gamma .
}
We first prove

\vspace{3mm}

\begin{theorem}
\label{t2}
The morphism $\eta_X$ is an equivalence of $G$-spectra for any $G$-spectrum $X$.
\end{theorem}

\Proof
As already mentioned, the theorem is a generalization of the ``Tate square", and the proof proceeds accordingly.
We use induction on $|G|$. The statement is clearly true for $|G|=1$, so assume
it is true with $G$ replaced by $G^\prime$, $|G^\prime|<|G|$.
Denote by $\check{P}(G)$ the partially ordered subset consisting of all sets $S\in P(G)$ such
that
$$G\notin S.$$

\vspace{3mm}
\begin{definition}
\label{ddd}
Denote by $\md$ the diagram
\beg{edd}{\diagram
& \widetilde{E\mf[G]}\wedge X\dto\\
\holim \Gamma|_{\check{P}(G)}\rto & \widetilde{E\mf[G]}\wedge\holim \Gamma|_{\check{P}(G)}.
\enddiagram
}
\end{definition}

\vspace{3mm}
Then transitivity of homotopy limits gives an equivalence
\beg{e4}{\holim \Gamma\r\holim \md.
}
Note that $E(G/G)=*$; in the diagram $\mathcal{D}$, the top term corresponds to $S=\{G\}$, the lower left term to
$G\notin S$ and the lower right term to $G\in S, |S|>1$.
Now for a subgroup $H\subsetneq G$, we have a canonical inclusion $P(H)\subseteq \check{P}(G)$.
If we consider 
$$\holim \Gamma|_{P(H)}$$
as a contravariant functor on the poset $Q$ of subgroups $H\subsetneq G$ with respect to inclusion,
we have a canonical equivalence
\beg{e5}{\diagram\protect\holim (\holim \Gamma{|}_{P(H)})\rto^(.6){\sim} & \protect\holim \Gamma{|}_{\check{P}(G)}
\enddiagram
}
where the outside homotopy limit on the left hand side of \rref{e5} is taken over
$Q$. This is true with $\Gamma$ replaced by any functor. By the induction hypothesis, however,
the canonical morphism
$$F(E(G/H)_+,X)\r\holim \Gamma_{P(H)}$$
is an equivalence for $H\subsetneq G$, so \rref{e5} yields a canonical equivalence
\beg{e6}{\diagram
F(E\mf[G]_+,X)=\holim F(E(G/?)_+,X)|_Q
\rto^(.7)\sim & \holim \Gamma|_{\check{P}(G)}.
\enddiagram
}
The first equality is by definition of $E\mf[G]$.
Therefore, if we denote by $\mathcal{E}$ the diagram
\beg{e7a}{\diagram
& \widetilde{E\mf[G]}\wedge X\dto\\
F(E\mf[G]_+,X)\rto & \widetilde{E\mf[G]}\wedge F(E\mf[G]_+,X),
\enddiagram
}
the canonical map
\beg{e7}{\holim \me\r\holim \md
}
is an equivalence, which further obviously commutes with the canonical morphisms from $X$.

Note, on the other hand, however, that the canonical morphism from $X$ to $\holim \mathcal{E}$
is an equivalence, since $\mathcal{E}$ is the generalized ``Tate square'' for the family $\mf[G]$.
In other words, the fiber of the canonical morphism 
$$X\r \widetilde{E\mf[G]}$$
maps to the fiber of the bottom row of $\mathcal{E}$ by the canonical equivalence
$$E\mf[G]_+\wedge X\r E\mf[G]_+\wedge F(E\mf[G]_+,X),$$
which is an equivalence.
\qed

\vspace{3mm}
To prove the ``non-derived'' statement \rref{et1} for $X=MU_G$, we will use induction,
which will have to involve a somewhat more general class of spectra. Concretely, by
{\em generalized $MU_G$} we mean the smallest class of $G$-equivariant spectra for all $G$ finite abelian
which satisfies the following:
\begin{enumerate}
\item
$MU_G$ is a generalized $MU_G$ for all $G$ finite abelian.

\item
If $R$ is a generalized $MU_G$, and $H\subsetneq G$, then
$$\Phi^H R$$
are generalized $MU_{G/H}$ where $\Phi^H(?)=(\widetilde{E\mf[H]}\wedge?)^H$ is
the ``geometric fixed point functor'' (see \cite{lms}, Definition 9.7).

\item 
If $R$ is a generalized $MU_G$, then
$$F(EG_+,R)$$
is a generalized $MU_G$.

\end{enumerate}

\vspace{3mm}
\begin{proposition}
\label{p1}
The completion theorem \cite{gm}, and the statements of Section 7 of \cite{Greenlees}
remain valid with $MU_G$ replaced by any generalized $MU_G$.
\end{proposition}

\Proof
Note that generalized $MU_G$'s are formed by starting with $MU_\Gamma$ for some
$\Gamma$ finite abelian, and then successively applying
\beg{ep1}{\Phi^H,
}
or
\beg{ep2}{F(EG_+,?)}
for certain subquotients $H,G$ of $\Gamma$. If only functors of the form \rref{ep1} are applied
in the process, iteration is in fact unnecessary, and we obtain an $MU_G$-algebra $R$
where $R_*$ is flat over $(MU_G)_*$ by a result of Greenlees (\cite{Greenlees}, Corollary 10.4).
Therefore, the proofs of \cite{gm} and \cite{Greenlees}, Section 7 apply verbatim with $MU_G$
replaced by $R$.

If, on the other hand, $R$ is a generalized $MU_G$ in whose formation a functor of the form
\rref{ep2} is used at least once, then the coefficients $R_*$ are known by a direct analogue of the computation
of Section \ref{s2} above. In particular, one sees explicitly that Euler classes of $1$-dimensional complex representations
still generate the augmentation ideal of $R_*$, and the proofs \cite{gm}, \cite{Greenlees},
Section 7, still apply with $MU_G$ replaced by $R$.
\qed

\vspace{3mm}

\noindent
{\bf Proof of Theorem \ref{t1}:} We will prove that the statement of Theorem \ref{t1}
is valid with $MU_G$ replaced by any generalized $MU_G$, which we will denote by $R$.
Our proof is by induction on $|G|$. For $|G|=1$, the statement is obvious. For a given
$|G|$, and $\{e\}\neq H\subseteq G$, denote 
first by $\mathcal{D}_H$ the subdiagram of $\Gamma$ on all objects of the form
$$ ...F(E(G/H^{\prime})_+,\widetilde{E\mf[H^\prime]}\wedge R)$$
where $H^\prime\supseteq H$, and by $\widehat{\mathcal{D}_H}$ the subdiagram of $\Gamma$ on all objects
of the form
$$... \widetilde{E\mf[H^\prime]}\wedge
F(EG/\{e\}_+,\widetilde{E\mf[\{e\}]}\wedge R)$$
where $H^\prime\supseteq H$. Note that both diagrams $\mathcal{D}_H$ and $\widehat{\mathcal{D}_H}$
are indexed by the subposet of $P(G)$ on all sets \rref{e1} with $H_1\supseteq H$,
which is isomorphic to $P(G/H)$. Now in the same sense as in the proof of Theorem \ref{t2}.

\vspace{3mm}
\begin{definition}
Denote 
by $\mathcal{M}_H$ the subdiagram of
$\Gamma$ of the form
\beg{e14}{\diagram
&\mathcal{D}_H\dto\\
F(EG/\{e\}_+,\widetilde{E\mf[\{e\}]}\wedge R)\rto & \widehat{\mathcal{D}_H}.
\enddiagram
} 
\end{definition}

\vspace{3mm}
Taking homotopy limits of the corners of $\mathcal{M}_H$ for a given $H\neq \{e\}$, we obtain the diagram
\beg{e15}{\diagram
& \widetilde{E\mf[H]}\wedge R\dto\\
F(EG_+,MU)\rto & \widetilde{E\mf[H]}\wedge F(EG_+,R),
\enddiagram
}
by the induction hypothesis applied to $G/H$, and the fact that 
$$\widetilde{E\mf[\{e\}]}=S^0.$$
Taking the diagram consisting of the union of the diagrams \rref{e15} over $H\neq \{e\}$ where we put the
canonical arrows between the corresponding upper right and lower right corners 
induced by inclusions of the subgroups $H$, is then equivalent to the
homotopy limit of
the diagram formed by taking the union of the diagrams $\mathcal{M}_H$,
which is the diagram $\Gamma$. On the other hand, taking homotopy limits over $H\neq \{e\}$
in the upper and lower right corners of \rref{e15}, we obtain the ``ordinary'' Tate square for $R$
(as considered for example in \cite{Greenlees}):
\beg{etate}{\diagram
& \widetilde{EG}\wedge R\dto\\
F(EG_+,R)\rto & \widetilde{EG}\wedge F(EG_+,R).
\enddiagram
}
Now by the induction hypothesis, the coefficients of the upper right and lower right corners of
\rref{e15} are equal to the inverse limits of the coefficient functor applied to
the corresponding parts of the diagram
\rref{e14}. On the other hand, consider the spectral sequences corresponding to the homotopy limits
of the upper right and lower right corners \rref{e15} whose $E_2$-terms are right derived
functors of the limits of the diagrams. By the first sentence of the proof
of Lemma 7.2 of Greenlees \cite{Greenlees}, which remains valid with $MU$ replaced by $R$ by Proposition \ref{p1}
above, the vertical arrow of \rref{e15} induces an {\em isomorphism} in filtration degrees $\geq 1$
of the $E_2$-terms of those spectral sequences, and hence these terms may be ignored, and we see that the corners 
of the (ordinary) Tate diagram for $R$ are obtained as non-derived limits of the corresponding
parts of the diagram $\Gamma$. 

Finally, the homotopy limit of the Tate square can only have a derived term in filtration degree $1$, but such a term
would create odd degree elements in $(MU_G)_*$, which do not exist by \cite{cm,l}.
\qed

\end{document}